\documentclass{article}
\usepackage{amsmath}
\usepackage{amsfonts}
\usepackage{amsthm}
\usepackage{amssymb}
\usepackage{color}

\newtheorem{theorem}{Theorem}[section]

\newtheorem{proposition}[theorem]{Proposition}

\newtheorem{problem}{Problem}

\newtheorem{corollary}[theorem]{Corollary}

\theoremstyle{definition}
\newtheorem{definition}[theorem]{Definition}
\theoremstyle{remark}

\theoremstyle{note}

\newcommand{\1}{\mathbf{1}}

\newcommand\remove[1]{}

\newcommand{\vf}{\varphi}

\def\f2{\mathbb{F}_2}

\newcommand{\ep}{\varepsilon}

\newcommand{\lb}{\label}

\newcommand{\buoo}{without loss of generality}

\newcommand{\de}{\delta}
\newcommand{\e}{\varepsilon}

\newcommand{\bbN}{\mathbb{N}}

\newcommand{\Om}{\Omega}

\begin{document}

\title{Bourgain discretization using Lebesgue-Bochner spaces}

\author{Mikhail~I.~Ostrovskii and Beata~Randrianantoanina}

\date{}
\maketitle

\begin{large}

{\sl This paper is dedicated to the memory of our friend Joe
Diestel (1943--2017). Lebesgue-Bochner spaces were one of the main
passions of  Joe. He started to work in this direction in his
Ph.D. thesis \cite{Die68}, and devoted to Lebesgue-Bochner spaces
a large part of his most popular, classical,
Dunford-Schwartz-style monograph \cite{DU77}, joint with Jerry
Uhl.}
\bigskip

\begin{abstract}
We study the Lebesgue-Bochner discretization property of Banach
spaces $Y$, which ensures that the Bourgain's discretization
modulus for $Y$ has a good lower estimate. We  prove that there
exist spaces that do not have the Lebesgue-Bochner discretization
property, and we give a class of examples of spaces that enjoy
this property.
\end{abstract}

\smallskip



\bigskip

\section{Introduction}

We denote by $c_Y(X)$ the greatest lower bound of distortions of
bilipschitz embeddings of a metric space $(X,d_X)$ into a metric
space $(Y,d_Y)$, that is, the greatest lower bound of the numbers
$C$ for which there is a map $f:X\to Y$ and a real number $r>0$
such that
\[\forall u,v\in X\quad r d_X(u,v)\le d_Y(f(u),f(v))\le r
Cd_X(u,v).\] See \cite{Mat02}, \cite{Nao18}, and \cite{Ost13} for
background on this notion. Let $X$ be a finite-dimensional Banach
space and $Y$ be an infinite-dimensional Banach space.

\begin{definition} For $\ep\in (0,1)$ let $\delta_{X\hookrightarrow Y}(\ep)$ be the supremum of those
$\delta\in (0,1)$ for which every $\delta$-net $\mathcal N_\delta$
in $B_X$ satisfies $c_Y(\mathcal{N}_\delta)\ge (1-\ep)c_Y(X)$.
 The function
$\delta_{X\hookrightarrow Y}(\ep)$ is called the
\emph{discretization modulus for embeddings of $X$ into $Y$}.
\end{definition}

It is not immediate that the discretization modulus is defined for
any $\ep\in(0,1)$, but this can be derived using the argument of
\cite{Rib76} and \cite{HM82} (see \cite[Introduction]{GNS12}).
Giving a new proof of the Ribe theorem \cite{Rib76} Bourgain
proved the following remarkable result \cite{Bou87} (we state it
in a stronger form which  was proved in \cite{GNS12}):

\begin{theorem}[Bourgain's discretization theorem]\label{T:BourgainDiscr} There exists $C\in (0,\infty)$ such that for every two Banach spaces $X,Y$ with $\dim X=n<\infty$ and $\dim Y=\infty$, and every $\ep\in (0,1)$, we have
\begin{equation}\label{E:BourgainDiscrImpr}
\delta_{X\hookrightarrow Y}(\ep)\ge e^{-(c_Y(X)/\ep)^{Cn}}.
\end{equation}
\end{theorem}

Bourgain's discretization theorem and the described below result
of \cite{GNS12} on improved estimates in the case of $L_p$ spaces
have important consequences for quantitative estimates of
$L_1$-distortion of the metric space consisting of finite subsets
(of equal cardinality) in the plane with the minimum weight
matching distance, see \cite[Theorem 1.2]{NS07}.\medskip

The proof of Bourgain's discretization theorem was clarified and
simplified in \cite{Beg99} and \cite{GNS12} (see also its
presentation in \cite[Section 9.2]{Ost13}). Different approaches
to proving Bourgain's discretization theorem in special cases were
found in \cite{LN13}, \cite{HLN16}, and \cite{HN16+}. However
these approaches do not improve the order of estimates for the
discretization modulus.
\medskip

On the other hand the paper \cite{GNS12} contains a proof with
much better estimates in the case where $Y=L_p$. The approach of
\cite{GNS12} is based on the following result (whose proof uses
   methods of \cite{JMS09}; origins of this approach can be found in \cite{GK03}).

\begin{theorem}[{\cite[Theorem 1.3]{GNS12}}]\label{T:LBdiscr}
There exists a universal constant $\kappa\in (0,\infty)$ with the
following property. Assume that $\delta,\ep\in (0,1)$ and $D\in
[1,\infty)$ satisfy $\delta\le \kappa\ep^2/(n^2D)$. Let $X,Y$ be
Banach spaces with $\dim X=n<\infty$, and let $\mathcal N_\delta$
be a $\delta$-net in $B_X$. Assume that $c_Y(\mathcal N_\delta)\le
D$. Then there exists a separable probability space
$(\Omega,\nu)$, a finite dimensional linear subspace $Z\subseteq
Y$, and a linear operator $T:X\to L_\infty(\nu,Z)$ satisfying
\begin{equation}\label{E:LBfactor}
\forall x\in X,\quad \frac{1-\ep}{D}\|x\|_X\le
\|Tx\|_{L_1(\nu,Z)}\le  \|Tx\|_{L_\infty(\nu,Z)}\le
(1+\ep)\|x\|_X.
\end{equation}
\end{theorem}

As is noted in \cite{GNS12}, since $(\Omega,\nu)$ is a probability
measure, we have
\[\|\cdot\|_{L_1(\nu,Z)}\le \|\cdot\|_{L_p(\nu,Z)}\le \|\cdot\|_{L_\infty(\nu,Z)}
\]
for every $p\in[1,\infty]$, therefore \eqref{E:LBfactor} implies
that $X$ admits an embedding into ${L_p(\nu,Z)}$ with distortion
$\le\displaystyle{\frac{D(1+\ep)}{1-\ep}}$. Since, by the
well-known Carath\'eodory theorem, $L_p(\nu,L_p)$ is isometric to
$L_p$ (see \cite[\S 14]{Lac74}) we get that if $Z$ is a subspace
of $L_p$, then $L_p(\nu,Z)$ is also a subspace of $L_p$, and, as
explained in \cite{GNS12}, it follows that the Bourgain's
discretization modulus for the case of $Y=L_p$ satisfies a much
better estimate
$$\de_{X\hookrightarrow L_p}(\ep)\ge \frac{\kappa \e^2}{n^{5/2}}$$
(since for all spaces $X, Y$ and all $\de>0$, $c_Y(\mathcal N_\delta)\le\sqrt{n}$, see \cite{GNS12}).
\medskip

To generalize this approach to a wider class of spaces it is
natural to introduce the following definition.

\begin{definition}\label{D:LBDiscr} We say that a Banach space $Y$ has {\it the
Lebesgue-Bochner discretization property} if for any separable
probability measure $\mu$, there exists a function
$f:[1,\infty)\to [1,\infty)$ so that for any $C\ge 1$ and any
finite dimensional subspace $Z\subset Y$, if $W$ is any
finite-dimensional subspace of $L_\infty(\mu,Z)$ such that for all
$w\in W$ \begin{equation}\label{E:LBdp}\|w\|_{L_\infty(\mu,Z)} \le
C\|w\|_{L_1(\mu,Z)},\end{equation} then $W$ is $f(C)$-embeddable
into $Y$.
\end{definition}

The following is a corollary of Theorem~\ref{T:LBdiscr}.

\begin{corollary} Let $Y$ be a Banach space with the
Lebesgue-Bochner discretization property, $\delta\le
\kappa\ep^2/(n^{5/2})$, where $\kappa$ is the constant of
Theorem~\ref{T:LBdiscr}, and $\mathcal{N}_\de$ be a $\delta$-net
in an $n$-dimensional Banach space $X$. Then
\begin{equation}\label{gdisc}
c_Y(X)\le g\left(\frac{1+\ep}{1-\ep}c_Y(\mathcal{N}_\de)\right),
\end{equation}
where $g(t):=tf(t)$ and $f$ is the function of
Definition~\ref{D:LBDiscr}. Thus if, for an increasing function
$g$, we define $\delta^g_{X\hookrightarrow Y}(\ep)$ as the
supremum of $\de$ so that \eqref{gdisc} is satisfied for all
$\de$-nets $\mathcal{N}_\de$ of $B_X$, we have that
$\delta^g_{X\hookrightarrow Y}(\ep)\ge \kappa\ep^2/(n^{5/2})$.
\end{corollary}

\begin{proof}
By Theorem~\ref{T:LBdiscr},  there exists a finite dimensional
subspace $Z\subset Y$ and a finite-dimensional subspace $W \subset
L_\infty(\nu,Z)$ (the image of the operator $T$) so that $W$
satisfies \eqref{E:LBdp} with
$C=\frac{1+\ep}{1-\ep}c_Y(\mathcal{N}_\de)$. Thus by the
Lebesgue-Bochner discretization property of $Y$,  $c_Y(W)\le
f(\frac{1+\ep}{1-\ep}c_Y(\mathcal{N}_\de))$, and we obtain
\[c_Y(X)\le\frac{1+\ep}{1-\ep}\ c_Y(\mathcal{N}_\de)f\left(\frac{1+\ep}{1-\ep}c_Y(\mathcal{N}_\de)\right).\qedhere\]
\end{proof}

\begin{problem}\label{P:LBdiscr} Characterize Banach spaces
with the
Lebesgue-Bochner discretization property.
\end{problem}

At the meeting of the Simons Foundation (New York City, February 20,
2015) Assaf Naor mentioned that at that time no examples of Banach
spaces which do not have the Lebesgue-Bochner discretization property were
known although people who were working on this (Assaf Naor and
Gideon Schechtman) believed that such examples should exist.
\medskip

We note that the based on the Fubini and Carath\'eodory theorems
argument showing that $L_p(L_p)$ is isometric to $L_p$ (for 
suitable measure spaces) fails for other functions spaces even in
a certain `isomorphic' form (see \cite[Appendix]{BBS02}).  For some
spaces a very strong opposite of the situation in the $L_p$-case
happens: Raynaud \cite{Ray89} proved that when $L_\vf([0,1],\mu)$
is an Orlicz space that is not isomorphic to some $L_p$ and does
not contain $c_0$ or $\ell_1$, then, for any $r\in[1,\infty)$  the
space $\ell_r(L_\vf)$ (and thus also $L_\vf([0,1],\mu,L_\vf)$) not
only does not embed in $L_\vf([0,1],\mu)$, but is not even crudely
finitely representable in it.

In general, if   $E$  is a Banach function space on a measure
space $(\Om,\mu)$, the structure of  the $E$-valued Bochner space
$E(\Om,\mu,E)$ can be very different from the structure of the
space $E$, see \cite{Rea90}, \cite{BBS02}, \cite{FPP08}. We refer
the reader to \cite{BBS02} for a detailed discussion and history
of related results.

In
this paper we show (Proposition \ref{P:NonSQNonLB}) that there is
a class of Banach spaces which do not have the Lebesgue-Bochner
discretization property and observe that this class contains the space
constructed by Figiel \cite{Fig72}.

 We also find some  examples, besides $L_p$,  of
Banach spaces that have the Lebesgue-Bochner discretization
property. An easy observation is that the Lebesgue-Bochner spaces
$L_p(E)$, where $E$ is any Banach space, have the Lebesgue-Bochner
discretization property. It is interesting that even  the finite
direct sums of such spaces   have the Lebesgue-Bochner
discretization property, see Proposition~\ref{propLpk}. We would
like to mention that many well-known and important spaces are of
the form $L_p(E)$. In particular,  the mixed norm Lebesgue spaces
$L^P$ introduced in \cite{BP61}   are such and thus have the
Lebesgue-Bochner discretization property. For
$P=(p_1,\dots,p_m)\in[1,\infty)^m$, the space $L^P$ consists of
measurable functions $f$ on $\Om=\prod_{i=1}^m (\Om_j,\mu_j)$, the
norm defined by
$$\|f\|_P:=
\left(\int\dots\left(\int \left(\int
|f(t_1,\dots,t_n)|^{p_1}d\mu_1\right)^{p_2/p_1}d\mu_2\right)^{p_3/p_2}\dots
d\mu_m\right)^{1/p_m}.$$ Mixed norm spaces of this type arise
naturally in harmonic and functional analysis. Such norms (and
their generalizations that use other function space norms in place
of the $L_{p_j}$-norms) are used for example to study Fourier and
Sobolev inequalities and embeddings of Sobolev spaces. The
properties and applications of mixed norm spaces are extensively
studied in the literature, see e.g. \cite{GS16,CS16,DPS10} and
their references.

\section{Finitely squarable Banach spaces}

\begin{definition}
An infinite-dimensional Banach space $Y$ is called {\it finitely
squarable} if there exists a constant $C$ such that
 for every
finite-dimensional subspace $Z\subset Y$ the direct sum
$Z\oplus_\infty Z$ admits a linear embedding into $Y$ with
distortion bounded by $C$.
\end{definition}

The first examples of Banach spaces which are not finitely
squarable were constructed by Figiel \cite{Fig72}. An easy
observation is that a Banach space $Y$ which is isomorphic to
$Y\oplus Y$, is finitely squarable. The converse it false. In
fact, both of the earliest examples of Banach spaces which are not
isomorphic to their squares, the James \cite{Jam50} quasireflexive
space $J$ \cite{BP60} and $c(\omega_1)$ \cite{Sem60} are finitely
squarable, and for very simple reason: they have trivial cotype.
For the James space this was proved in \cite{GJ73}, for
$c(\omega_1)$ this is obvious. Modern Banach space theory provides
much more sophisticated examples of finitely squarable spaces
which are not isomorphic to their squares, for example, the
Argyros-Haydon space \cite{AH11}.

\begin{proposition}\label{P:NonSQNonLB} Any space which is not finitely squarable does not
have the Lebes\-gue-Bochner discretization property.
\end{proposition}

\begin{proof} Let $Z$ be a subspace of $Y$ for which $Z\oplus_\infty Z$ is ``very far'' from a
subspace of $Y$.

We introduce the following subspace $W\subset L_\infty ([0,1],Z)$:
it consists of all $Z$-valued functions which are constant on the
first half and constant on the second half, but these constants
can be different vectors of $Z$. It is clear that this space is
isometric to $Z\oplus_\infty Z$. It is also clear that the
$L_1([0,1],Z)$ norm on this subspace is $2$-equivalent to the
$L_\infty$-norm. The conclusion follows.
\end{proof}

This proposition makes the following problem important:

\begin{problem}\label{P:FinNonSq} Does there exist a finitely squarable space which
does not have the Lebesgue-Bochner discretization property?
\end{problem}

We conjecture that the answer to Problem \ref{P:FinNonSq} is
positive.

\section{Examples of spaces with the Lebesgue-Bochner discretization property}

In this section we provide some examples of spaces having the
Lebesgue-Bochner discretization property. In all proofs below we
use the notation of Definition~\ref{D:LBDiscr}, that is: $Y$ is a
Banach space, $C>0$,  $Z\subset Y$ is a finite dimensional
subspace of $Y$. Since we consider separable probability measures,
by the Carath\'eodory theorem \cite{Lac74} we may assume that $W$
is a finite-dimensional subspace of $L_\infty([0,1],\mu,Z)$, such
that for all $w\in W$
\begin{equation}\lb{normequiv}
\frac1C \|w\|_{L_\infty([0,1],\mu,Z)} \le\|w\|_{L_1([0,1],\mu,Z)}\le
\|w\|_{L_\infty([0,1],\mu,Z)}.
\end{equation}

Since $W$ is finite dimensional,  for any $\e>0$,   there exists a subspace
$\tilde{W}\subseteq L_\infty([0,1],\mu,Z)$ with Banach-Mazur distance from $W$ less than $1+\e$, so that $\tilde{W}$ is spanned by simple
functions and all $w\in \tilde{W}$ satisfy \eqref{normequiv} with $C$ replaced by
$(1+\e)C$. Thus, \buoo, we
 may    assume that $W$ is spanned by simple
functions which are constant on elements $\{\Delta_i\}_{i=1}^n$ of
some partition of $[0,1]$ into sets of measure $\frac1n$. Thus we
can denote elements $w\in W$ as
\begin{equation*}
w=(w_1,\dots,w_n),
\end{equation*}
meaning that  $w=\sum_{i=1}^n\1_{\Delta_i}\otimes w_i$. For all
$w\in W$ we have $\|w\|_{L_\infty([0,1],\mu,Z)}=\max_{1\le i\le n}
\|w_i\|_Z$, and for all $p$, $1\le p<\infty$, we have
$$\|w\|_{L_p([0,1],\mu,Z)}
=\Big(\frac 1n \sum_{i=1}^n \|w_i\|_Z^p\Big)^{\frac1p}.$$

Given any $p\in[1,\infty]$, $k\in \bbN$, and any Banach spaces
$E_1,\dots, E_k$, by $L_p^k(E_1,\dots, E_k)$ we denote the Banach
space of all $k$-tuples $(a_1,\dots,a_k)$ such that $a_j\in E_j$
for all $j\in[k]$, endowed with the norm
$$\|(a_1,\dots,a_k)\|_{L_p^k(E_1,\dots, E_k)}:=
\Big(\frac 1k \sum_{i=1}^k \|a_i\|_Z^p\Big)^{\frac1p}, \hbox{ if
 }p<\infty,$$
$$\|(a_1,\dots,a_k)\|_{L_\infty^k(E_1,\dots, E_k)}:=\max_{1\le
i\le k}\|a_i\|_Z.$$ If the spaces $E_1,\dots,E_k$ are equal to the
same space $E$, we denote $L_p^k(E_1,\dots, E_k)$ by $L_p^k(E)$.

\begin{proposition}\label{propLpk} Let $k\in \bbN$, $p,q_1,\dots,q_k\ge 1$, $X_1, \dots,X_k$
be  any Banach spaces,  and  for each $j\in[k]$ let
$(\Om_j,\mu_j)$ be any nonatomic separable measure space, with
finite or infinite measure, or $\Om_j=\bbN$ and $\mu_j$ is the
counting measure. Then the space
$$Y=L_p^k(L_{q_1}(\Om_1,\mu_1,X_1), L_{q_2}(\Om_2,\mu_2,X_2),\dots,L_{q_k}(\Om_k,\mu_k,X_k))$$
has the Lebesgue-Bochner discretization property with $f(C)\le
k^{2-1/p}C$.
\end{proposition}

Note that since the constant $f(C)$ in Definition \ref{D:LBDiscr}
can depend on $k$, the fact that $Y$ is an $L_p^k$-sum is not
essential, essential is the fact that $Y$ is a finite direct
sum.

\begin{proof}[Proof of Proposition \ref{propLpk}]
To simplify notation we will omit the measure spaces when writing
the symbol for a Lebesgue-Bochner space, i.e. we will write
$L_{q_j}(X_j)$ instead of $L_{q_j}(\Om_j,\mu_j,X_j)$ with the
understanding that for all $j\in[k]$, the measure spaces are those
fixed in the statement of the proposition.

Since for any $p\ge 1$, the space $L_p^k(L_{q_1}(X_1),
L_{q_2}(X_2),\dots,L_{q_k}(X_k))$ is $k^{1-1/p}$-isomor\-phic to
$L_1^k(L_{q_1}(X_1), L_{q_2}(X_2),\dots,L_{q_k}(X_k))$, it is
enough to prove that in the case where $p=1$ we have $f(C)\le kC$.
Note that if at least one of $q_j$ is equal to $\infty$, the space
$Y$ has trivial cotype and thus has the Lebesgue-Bochner
discretization property. In the following we assume that
$q_j<\infty$ for all $j\in[k]$.

Using the discussion and notation preceding Proposition~\ref{propLpk}, we see that it suffices to prove that any subspace
$W\subseteq L_\infty^n(Y)$ satisfying
\begin{equation}\label{normequiv2}
\forall w\in W\quad \quad  \frac1C \|w\|_{L_\infty^n(Y)} \le
\|w\|_{L_1^n(Y)}
\end{equation}
admits a $kC$-isomorphic embedding into $Y$.

Let $n\in\bbN$, $w=(w_1,\dots,w_n)\in W\subseteq L^n_\infty(Y)$ and, for $i\in[n]$,
$w_i=(w_{ij})_{j=1}^k\in L_1^k(L_{q_1}(X_1), L_{q_2}(X_2),\dots,L_{q_k}(X_k))$, where, for all $i\in[n]$, $j\in[k]$, $w_{ij}\in L_{q_j}(X_j)$. We will define a map
$\vf$ from $L^n_\infty(Y)$ to $Y$ so that for all $w\in L^n_\infty(Y)$ we have
\begin{equation}\lb{normvfw}
 \begin{split}
\|\vf(w)\|_{Y}=\frac 1k\sum_{j=1}^k \Big(\frac 1n\sum_{i=1}^n \|w_{ij}\|^{q_j}_{q_j}\Big)^{\frac1{q_j}}.
 \end{split}
 \end{equation}

 For each $j\in[k]$, we select $n$ mutually disjoint subsets
 $\{\Om_{j\nu}\}_{\nu=1}^n$ of $\Om_j$ so that for each $\nu\in[n]$ there exists a constant
 $a_{j\nu}>0$ and a surjective isometry  $T_{j\nu}:L_{q_j}(\Om_j,\mu_j)\to L_{q_j}(\Om_{j\nu},a_{j\nu}\mu_j)$.   It is well-known that when $(\Om_j,\mu_j)$ is nonatomic or equal to $\bbN$ with the counting measure then such choices are possible, and that the isometry $T_{j\nu}$ can be naturally extended to the isometry
 $\bar{T}_{j\nu}$ from the Lebesgue-Bochner space $L_{q_j}(\Om_j,\mu_j,X_j)$
 onto $L_{q_j}(\Om_{j\nu},a_{j\nu}\mu_j,X_j)$, cf. e.g. \cite{DU77}.

We define the map $\vf_j: L^n_\infty(L_{q_j}(X_j),\dots,
L_{q_j}(X_j))\to L_{q_j}(X_j)$ by setting for all
 $(x_1,\dots,x_n)\in  L^n_\infty(L_{q_j}(X_j),\dots,L_{q_j}(X_j))$
\begin{equation*}
 \begin{split}
\vf_j(x_1,\dots,x_n):=\sum_{\nu=1}^n \left(\frac
{a_{j\nu}}{n}\right)^{\frac1{q_j}}\bar{T}_{j\nu}x_\nu.
 \end{split}
 \end{equation*}

 Since the sets $\{\Om_{j\nu}\}_{\nu=1}^n$ are mutually disjoint, we get
 \begin{equation}\lb{normvfj}
 \begin{split}
\|\vf_j(x_1,\dots,x_n)\|_{q_j}=\Big(\frac 1n\sum_{\nu=1}^n
a_{j\nu}\|\bar{T}_{j\nu}x_\nu\|_{q_j}^{q_j}\Big)^{\frac1{q_j}}=
 \Big(\frac 1n\sum_{\nu=1}^n \|x_\nu\|_{q_j}^{q_j}\Big)^{\frac1{q_j}}.
 \end{split}
 \end{equation}

 Next, given  $w=(w_1,\dots,w_n)\in  L^n_\infty(Y)$ where, for $i\in[n]$,
$w_i=(w_{ij})_{j=1}^k\in Y= L_1^k(L_{q_1}(X_1), L_{q_2}(X_2),\dots,L_{q_k}(X_k))$, we define $\vf(w)\in Y$ by setting
\begin{equation*}
 \begin{split}
\vf((w_i)_{i=1}^n)&=\Big(\vf_j\big((w_{i j})_{i=1}^n\big)\Big)_{j=1}^k.
 \end{split}
 \end{equation*}

By \eqref{normvfj}, equality \eqref{normvfw} is satisfied.

We will show that for all $w\in W\subseteq L^n_\infty(Y)$, we have
\begin{equation}\lb{goal}
 \begin{split}
\|w\|_{L^n_1(Y)}\le \|\vf(w)\|_{Y}\le k\|w\|_{L^n_\infty(Y)}.
 \end{split}
 \end{equation}

To prove the leftmost inequality we write
\begin{equation*}
 \begin{split}
 \|w\|_{L^n_1(Y)}&=\frac 1n\sum_{i=1}^n \|w_i\|_{Y} =\frac 1n\sum_{i=1}^n \Big(\frac 1k\sum_{j=1}^k \|w_{ij}\|_{q_j}\Big)\\
 &=\frac 1k\sum_{j=1}^k \Big(\frac 1n\sum_{i=1}^n \|w_{ij}\|_{q_j}\Big)
 \le\frac 1k\sum_{j=1}^k \Big(\frac 1n\sum_{i=1}^n \|w_{ij}\|^{q_j}_{q_j}\Big)^{\frac1{q_j}}\\
  &=\|\vf(w)\|_{Y},
 \end{split}
 \end{equation*}
where the inequality follows from the classical theorem on
averages (\cite[Theorem 16]{HLP52}) applied to each of the
summands with the corresponding exponent $q_j$, for $1\le j\le k$,
respectively.

To prove the rightmost inequality, for each $j\in[k]$, let
$i_j\in[n]$ be such that $\|w_{i_jj}\|_{q_j}=\max_{1\le i\le n}
\|w_{ij}\|_{q_j}$. Then
\begin{equation*}
 \begin{split}
 \|\vf(w)\|_{Y}&=\frac 1k\sum_{j=1}^k \Big(\frac 1n\sum_{i=1}^n \|w_{ij}\|^{q_j}_{q_j}\Big)^{\frac1{q_j}}\le \frac 1k\sum_{j=1}^k \Big(\max_{1\le i\le n} \|w_{ij}\|_{q_j}\Big) \\
 &= \frac 1k\sum_{j=1}^k  \|w_{i_jj}\|_{q_j}\le \frac 1k\sum_{j=1}^k \Big( \sum_{l=1}^k \|w_{i_jl}\|_{q_j}\Big)\\
 &= \sum_{j=1}^k \Big(\frac 1k \sum_{l=1}^k \|w_{i_jl}\|_{q_j}\Big)
 =\sum_{j=1}^k \|w_{i_j}\|_{Y}\\
 &\le k \max_{1\le i\le n} \|w_{i}\|_{Y}=k\|w\|_{L^n_\infty(Y)}.
 \end{split}
 \end{equation*}

 Thus, \eqref{goal} holds, and therefore,  by \eqref{normequiv2},
 \begin{equation*}
 \begin{split}
\frac1C\|w\|_{L^n_\infty(Y)}\le \|\vf(w)\|_{Y}\le
k\|w\|_{L^n_\infty(Y)}.\qedhere
 \end{split}
 \end{equation*}
\end{proof}
\bigskip

\thanks{ \textbf{Acknowledgements:}
We would like to thank Assaf Naor for suggesting the problem, and
Florence Lancien for valuable discussions.

The first named author was supported by the National Science
Foundation under Grant Number DMS--1700176.}

\end{large}

\begin{small}

\renewcommand{\refname}{
\end{small}

\textsc{Department of
Mathematics and Computer Science, St. John's University, 8000 Utopia Parkway, Queens, NY 11439, USA} \par
  \textit{E-mail address}: \texttt{ostrovsm@stjohns.edu} \par
  \medskip

\textsc{Department of Mathematics, Miami University,
Oxford, OH 45056, USA} \par
  \textit{E-mail address}: \texttt{randrib@miamioh.edu} \par

\end{document}